\documentclass[11pt]{article}
\usepackage[utf8]{inputenc}
\usepackage[T1]{fontenc}
\usepackage[english]{babel}
\usepackage{microtype}
\usepackage[margin=1in]{geometry}

\usepackage{amsmath,amssymb,amsthm,mathrsfs,mathtools}
\usepackage[dvipsnames]{xcolor}
\usepackage{booktabs,longtable,array,tabularx,multirow}
\usepackage{tikz,pgfplots}
\pgfplotsset{compat=1.18}
\usetikzlibrary{arrows.meta,positioning,decorations.pathreplacing,calligraphy,calc,fit}

\usepackage{enumitem}
\usepackage{adjustbox}
\usepackage{graphicx}
\usepackage{subcaption}
\usepackage{float}
\allowdisplaybreaks
\setlist[enumerate,1]{label={\rm(\roman*)}}
\setlist[enumerate,2]{label={\rm(\alph*)}}
\setlist[enumerate,3]{label={\rm(\arabic*)}}
\usepackage[hypertexnames=false]{hyperref}
\hypersetup{
    colorlinks,
    linkcolor={red!60!black},
    citecolor={green!60!black},
    urlcolor={blue!60!black}
}
\newcommand{\Z}{\mathbb{Z}}
\newcommand{\Q}{\mathbb{Q}}

\newcommand{\Sym}{\mathfrak{S}}
\newcommand{\Canon}{\mathcal{C}}
\newcommand{\Lat}{\mathcal{L}}
\newcommand{\Dyck}{\operatorname{Dyck}}
\newcommand{\FussDyck}{\operatorname{FussDyck}}

\newcommand{\Tree}{\operatorname{Tree}}
\newcommand{\Match}{\operatorname{Match}}

\newcommand{\Des}{\operatorname{Des}}
\newcommand{\des}{\operatorname{des}}

\newcommand{\val}{\operatorname{val}}

\newcommand{\firstpos}{\operatorname{first}}
\newcommand{\lastpos}{\operatorname{last}}
\newcommand{\rev}{\operatorname{rev}}
\newcommand{\comp}{\operatorname{comp}}
\newcommand{\rc}{\operatorname{rc}}
\newcommand{\Orb}{\operatorname{Orb}}
\newcommand{\skel}{\operatorname{skel}}
\newcommand{\card}[1]{\lvert #1\rvert}

\newcommand{\Bsf}{\mathsf{B}}
\newcommand{\Qsf}{\mathsf{Q}}
\newcommand{\FC}{\operatorname{FC}}
\newcommand{\ind}{\mathbf{1}}

\newtheorem{theorem}{Theorem}[section]
\newtheorem{proposition}[theorem]{Proposition}
\newtheorem{lemma}[theorem]{Lemma}
\newtheorem{corollary}[theorem]{Corollary}
\newtheorem{conjecture}[theorem]{Conjecture}
\newtheorem{question}[theorem]{Question}

\theoremstyle{definition}
\newtheorem{definition}[theorem]{Definition}
\newtheorem{example}[theorem]{Example}

\newtheorem*{remark}{Remark}
\numberwithin{equation}{section}

\title{Avoiding patterns with three distinct letters in Canon Permutations}

\author{
Umesh Shankar\\
    Department of Computer Science and Automation,\\
    Indian Institute of Science Bengaluru,\\
    Bengaluru 560012, Karnataka, India\\
    Email: \texttt{umeshshankar@outlook.com}
}
\date{\today}

\begin{document}
\maketitle

\begin{abstract}
We study avoidance of patterns of length $3$ with three distinct letters in canon permutations. We reduce the problem to studying pattern avoidance in lattice words and show that there are $6$ such pattern avoiding classes. This shows that there are $12$ classes for the original Canon permutation pattern avoidance problem. We also give descent refinements for these classes and classify the patterns for which the descent enumeration gives palindromic and $\gamma$-positive polynomials. When the polynomials are $\gamma$-positive, we explain the $\gamma$-positivity through a group action analogous to Foata-Strehl valley hopping. Additionally, we study the avoidance of patterns in the relabelling orbit of $1213, 12112, 1231$ after a conjecture about their cardinalities by Laudone and give bijective proofs for the results.
\end{abstract}




\section{Introduction}
Let
\[
\mathcal M_n^k=\{1^k,2^k,\ldots,n^k\}
\]
be the multiset containing $k$ copies of every letter in $[n]$.  A
\emph{canon permutation} is a permutation of $\mathcal M_n^k$ in which the
$j$th copies for every $j\in[k]$, all occur in the same relative order.
Elizalde \cite{ElizaldeNonnesting,ElizaldeCanon} introduced these objects as the natural higher-multiplicity extension
of nonnesting multipermutations and showed that their descent polynomials are connected to both the Eulerian and Narayana polynomials. Beck and Deligeorgaki \cite{BeckDeligeorgaki} subsequently interpreted canon permutations as labeled linear
extensions of the product poset $[k]\times[n]$ and explained the resulting
palindromicity and $\gamma$-positivity by $(P,\omega)$-partition and
Foata--Strehl-type methods.

While there is a vast history and literature concerning the study of Wilf-equivalence (see the monograph by Kitaev \cite{KitaevBook} on the subject), the study of pattern avoidance in non-nesting permutations was done by Elizalde and Luo \cite{ElizaldeLuo}. Recently, Laudone \cite{Laudone} initiated the systematic study of classical pattern avoidance inside canon permutations.  A central theme of that work is that
avoidance sets closed under relabeling can be transferred from canon
permutations to regular lattice words. Laudone asks for a complete classification of classes of Canon permutations avoiding patterns of length $3$ with three distinct letters.  

A canon permutation on
$\{1^k,2^k,3^k\}$ is encoded by an underlying permutation
$\sigma\in\Sym_3$ and a standard Young tableau of shape $(k,k,k)$, or,
equivalently, a three-row lattice word.  Relabeling the rows by $\sigma$
translates the forbidden set $\Lambda$ to $\sigma^{-1}\Lambda$. This lets us recast the problem into one concerning pattern avoiding lattice words. The first main theorem in our work is the following.

\begin{theorem}
\label{thm:intro-six-state}
For every $k\ge 1$ and for every $\Lambda\subseteq\Sym_3$ and every fixed underlying permutation
$\sigma\in\Sym_3$, the cardinality of
$\Canon_3^{k,\sigma}(\Lambda)$ is one of
\[
0,\qquad 1,\qquad C_k,\qquad B_k,\qquad Q_k,\qquad T_k.
\]
Here $C_k$ is the Catalan number, $T_k=f^{(k,k,k)}$ is the unrestricted rectangular-tableau number,
$B_k$ counts three-row lattice words in which every $1$ precedes every $3$,
and $Q_k$ counts $321$-avoiding three-row lattice words.  Summing the six
components gives exactly twelve distinct formulas for the sixty-four subsets of
$\Sym_3$.
\end{theorem}

In particular, the four components for $\{213,312\}$ are
indexed by
\(
123,132,231,321,
\)
and each is a Dyck-path class. The descent refinement for the underlying permutation sets gives
\(
N_k(x), xN_k(x)\)
\\\( xN_k(x), x^2N_k(x),
\)
so that
\[
D_{\{213,312\},k}(x)=(1+x)^2N_k(x).
\]
The class avoiding $\{123,321\}$ instead has descent polynomial
$4xN_k(x)$ and so, we see that patterns in the same class do not give the same descent enumerating polynomials.

Next, we address palindromicity and $\gamma$-positivity.
Apart from the unrestricted class, every $\gamma$-positive example
is a Narayana class with zero, one, or two independent binary choices for the underlying permutations, or a
degenerate class with polynomial $1+x$, $1$, or $0$. For the $\gamma$-positive restricted classes,
the standard two-coloured Motzkin encoding of Dyck paths, together with the
external binary choices, gives a direct Boolean-orbit interpretation of every
$\gamma$-coefficient.

Finally, we give some enumeration results for the other classes. The sequence
$B_k$ (\href{https://oeis.org/A274969}{A274969}) begins
\[
B_k=1,4,21,121,728,\ldots
\]
and has a closed binomial formula, a cubic algebraic generating function,
and an independent interpretation by constrained pairs of binary trees \cite{BorieFalque}.  The
sequence $Q_k$ (\href{https://oeis.org/A284733}{A284733}) whose first few entries are
\[
Q_k=1,5,33,234,1706,12618,\ldots
\]
is identified with a rectangular-poset avoidance sequence.  We derive an
exact ballot-path multisum and hence prove P-recursiveness.

Laudone additionally makes three conjectures concerning the relabeling
orbits of the repeated-letter patterns
\(
1213, 12112, 1231.
\)
If $\Orb(\tau)$ denotes
the orbit of a pattern $\tau$ under arbitrary relabeling of its distinct
letters, then symmetric avoidance removes the choice of the underlying
permutation and leaves a lattice-word problem.  The three conjectural families have three different combinatorial explanations.

\begin{theorem}\label{thm:intro-three-decompositions}
The following statements hold.
\begin{enumerate}
\item For $n\geq2$ and $k\geq1$, a lattice word avoids $\Orb(1213)$ if and only if it consists of the
forced blocks
\[
1^k2^k\cdots(n-2)^k
\]
followed by an arbitrary binary lattice word on $\{n-1,n\}$.  Consequently,
\[
 c_n^k(\Orb(1213))=n!C_k.
\]
\item For $n\geq1$ and $k\geq2$, a lattice word avoids $\Orb(12112)$ if and only if deleting the final
occurrence of every letter leaves
\[
1^{k-1}2^{k-1}\cdots n^{k-1}.
\]
Replacing each non-final occurrence by \(U=(1,1)\) and each final occurrence by \(D=(1,-(k-1))\) gives a bijection with $(k-1)$-Fuss--Dyck paths
having $n$ down-steps. Hence
\[
 c_n^k(\Orb(12112))
 =n!\FC_{k,n}
 =\frac{n!}{(k-1)n+1}\binom{kn}{n}.
\]
\item For $n,k\geq1$, a lattice word avoids $\Orb(1231)$ if and only if the overlap graph of
the first-to-last occurrence intervals is a matching in the path $P_n$.
Equivalently, the word is a unique concatenation of singleton blocks $i^k$
and two-letter blocks indexed by matched edges $\{i,i+1\}$, each two-letter
block being a non-mountain Dyck word. Therefore
\[
 c_n^k(\Orb(1231))
 =n!\sum_{j=0}^{\lfloor n/2\rfloor}
 \binom{n-j}{j}(C_k-1)^j.
\]
\end{enumerate}
\end{theorem}

The paper is organized as follows. In Section~\ref{sec:preliminaries}, we fix the
pattern, orbit, lattice-word, tableau, and descent conventions.
In Sections~\ref{sec:three-letter-structure} and~\ref{sec:complete-enumeration}, we prove the
six-class theorem and the complete three-letter classification.
Section~\ref{sec:gamma} develops the descent refinements and the exact
$\gamma$-positivity classification.  Sections~\ref{sec:Bstate} and
\ref{sec:Qstate} study the two non-Catalan classes.
In Section~\ref{sec:symmetric-families}, we prove the three general structural
decompositions in Theorem~\ref{thm:intro-three-decompositions}.
All conjectures and open
questions are collected in
Section~\ref{sec:questions}.
\section{Definitions and preliminaries}\label{sec:preliminaries}

\subsection{Patterns and relabeling orbits}

A \emph{pattern} is a word $\tau=\tau_1\cdots\tau_m$ whose set of letters is
$[r]$ for some $r\leq m$.  Two words $a_1\cdots a_m$ and
$b_1\cdots b_m$ are \emph{order-isomorphic} if
\[
a_i<a_j\Longleftrightarrow b_i<b_j,
\qquad
a_i=a_j\Longleftrightarrow b_i=b_j
\]
for every $i,j$.  A word $w=w_1\cdots w_N$ \emph{contains} $\tau$ if some
subsequence $w_{i_1}\cdots w_{i_m}$ is order-isomorphic to $\tau$ and we say that it \emph{avoids} $\tau$ if it contains no such subsequence.  Avoidance of a set of patterns means simultaneous
avoidance of all members of the set.

The symmetric group $\Sym_r$ acts on patterns with alphabet $[r]$ by relabeling them:
for $\rho\in\Sym_r$,
\[
\rho\cdot(\tau_1\cdots\tau_m)
 =\rho(\tau_1)\cdots\rho(\tau_m).
\]
We will write
\[
\Orb(\tau)=\{\rho\cdot\tau:\rho\in\Sym_r\}.
\]
Thus $\Orb(1213)$ consists of the patterns $abac$ with $a,b,c$ pairwise
distinct, while $\Orb(1231)$ consists of the patterns $abca$ with
$a,b,c$ pairwise distinct.  Similarly,
\[
\Orb(12112)=\{12112,21221\}.
\]
A set of patterns is said to be \emph{symmetric} if it is a union of relabeling orbits.

When every letter is distinct, patterns are ordinary permutations.
For the three-letter part of the paper we use the fixed order
\(
123,132,213,231,312,321
\)
for the elements of $\Sym_3$.

\subsection{Canon permutations, tableaux, and lattice words}

A permutation of $\mathcal M_n^k$ is also called a \emph{$k$-regular word on
$[n]$}.  The $j$th copy of a letter means its $j$th occurrence from the left.

\begin{definition}
A $k$-regular word $\pi$ on $[n]$ is a \emph{canon permutation} if there is a permutation $\sigma=\sigma_1\cdots\sigma_n\in\Sym_n$ such that, for every
$j\in[k]$, the subsequence formed by the $j$th copies of the letters is
$\sigma$.  We call $\sigma$ the \emph{underlying permutation} and write
$\sigma_\pi=\sigma$.
\end{definition}

Let $\Canon_n^k$ be the set of canon permutations and
$\Canon_n^{k,\sigma}$ the fixed-underlying subset.  For a pattern set $\Lambda$,
we write
\[
\Canon_n^k(\Lambda)
 =\{\pi\in\Canon_n^k:\pi\text{ avoids }\Lambda\},
\qquad
c_n^k(\Lambda)=\card{\Canon_n^k(\Lambda)},
\]
and we use analogous notation with the superscript $\sigma$.

A word $w=w_1\cdots w_{nk}$ on $[n]$ is a \emph{$k$-regular lattice word} if
each letter occurs $k$ times and every prefix satisfies
\begin{equation}\label{eq:general-lattice-prefix}
\#_1\geq\#_2\geq\cdots\geq\#_n.
\end{equation}
Let $\Lat_n^k$ denote this set.  The row-word construction gives a bijection
between $\Lat_n^k$ and standard Young tableaux of rectangular shape $(k^n)$:
entry $t$ lies in row $w_t$.  In particular,
\[
\card{\Lat_n^k}=f^{(k^n)}.
\]

For $\sigma\in\Sym_n$, we define
\begin{equation}\label{eq:Phi-general}
\Phi_\sigma(w_1\cdots w_{nk})
 =\sigma_{w_1}\sigma_{w_2}\cdots\sigma_{w_{nk}}.
\end{equation}
The map $\Phi_\sigma$ is a bijection from $\Lat_n^k$ to
$\Canon_n^{k,\sigma}$.

For $w\in\Lat_n^k$, let $p_{i,r}(w)$ be the position of the $r$th occurrence of
$i$.

\begin{lemma}\label{lem:ordered-occurrences}
For every $w\in\Lat_n^k$,
\[
p_{i,r}(w)<p_{j,r}(w)
\qquad(1\leq i<j\leq n,\ 1\leq r\leq k).
\]
A $k$-regular word belongs to $\Lat_n^k$ if and only if these
inequalities hold for all adjacent pairs $j=i+1$.
\end{lemma}

\begin{proof}
If $p_{j,r}<p_{i,r}$ for some $i<j$, then immediately after the $r$th copy of
$j$ appears, the prefix contains at least $r$ copies of $j$ but at most $r-1$
copies of $i$.  This contradicts the prefix inequalities
\eqref{eq:general-lattice-prefix}.  Conversely, suppose
$p_{i,r}<p_{i+1,r}$ for all $i,r$.  At any prefix, if $a_{i+1}$ copies of
$i+1$ have occurred, then the first $a_{i+1}$ copies of $i$ have already
occurred, so $a_i\geq a_{i+1}$.  Hence every prefix satisfies
\eqref{eq:general-lattice-prefix}.
\end{proof}

\subsection{Symmetric avoidance reduction}

For a symmetric pattern set $\Lambda$, define
\[
\Lat_n^k(\Lambda)
 =\{w\in\Lat_n^k:w\text{ avoids }\Lambda\},
\qquad
L_{n,k}(\Lambda)=\card{\Lat_n^k(\Lambda)}.
\]
The following symmetric decomposition appears in Laudone's work.
\begin{theorem}\label{thm:symmetric-reduction}
Let $\Lambda$ be a symmetric set of patterns.  For every
$\sigma\in\Sym_n$ and $w\in\Lat_n^k$,
\[
\Phi_\sigma(w)\text{ avoids }\Lambda
\quad\Longleftrightarrow\quad
w\text{ avoids }\Lambda.
\]
Consequently, every underlying-permutation component has the same size and
\begin{equation}\label{eq:symmetric-factor}
 c_n^k(\Lambda)=n!\,L_{n,k}(\Lambda).
\end{equation}
\end{theorem}

\begin{proof}
The map $\Phi_\sigma$ changes only the names of the row letters.  Hence a
subsequence of $w$ with equality pattern and relative-order pattern $\tau$
is sent to a subsequence whose pattern is obtained from $\tau$ by a relabeling of its distinct letters.  Since $\Lambda$ is a union of complete relabeling
orbits, the selected subsequence belongs to $\Lambda$ before relabeling if and
only if it belongs to $\Lambda$ afterward.  This proves the avoidance
equivalence.  The maps $\Phi_\sigma$ are bijections and the fixed-underlying
sets partition $\Canon_n^k$, so summing the common component size over the $n!$
choices of $\sigma$ gives \eqref{eq:symmetric-factor}.
\end{proof}

\subsection{The three-letter relabeling reduction}

For $\Lambda\subseteq\Sym_3$ and $\sigma\in\Sym_3$, regard $\sigma$ as the
bijection $i\mapsto\sigma_i$ and put
\[
\sigma^{-1}\Lambda
 =\{\sigma^{-1}\circ\tau:\tau\in\Lambda\}.
\]
For $\Omega\subseteq\Sym_3$, abbreviate
\[
L_k(\Omega)=\card{\Lat_3^k(\Omega)}.
\]
For nonsymmetric pattern sets, the reduction retains the dependence on the underlying permutation.
\begin{theorem}
\label{thm:fixed-underlying-reduction}
For $\sigma\in\Sym_3$, $w\in\Lat_3^k$, and
$\Lambda\subseteq\Sym_3$,
\[
\Phi_\sigma(w)\text{ avoids }\Lambda
\quad\Longleftrightarrow\quad
w\text{ avoids }\sigma^{-1}\Lambda.
\]
Consequently,
\begin{equation}\label{eq:fixed-count-reduction}
 c_3^{k,\sigma}(\Lambda)=L_k(\sigma^{-1}\Lambda)
\end{equation}
and
\begin{equation}\label{eq:global-count-reduction}
 c_3^k(\Lambda)
 =\sum_{\sigma\in\Sym_3}L_k(\sigma^{-1}\Lambda).
\end{equation}
\end{theorem}

\begin{proof}
An occurrence of a length-three permutation pattern in a word on
$\{1,2,3\}$ necessarily uses one copy of each letter.  If
$w_{i_1}w_{i_2}w_{i_3}$ has pattern $\rho$, then
$\sigma_{w_{i_1}}\sigma_{w_{i_2}}\sigma_{w_{i_3}}$ has pattern
$\sigma\circ\rho$.  Thus it belongs to $\Lambda$ exactly when
$\rho\in\sigma^{-1}\Lambda$.  The first equivalence follows, and
\eqref{eq:fixed-count-reduction} follows from the bijectivity of
$\Phi_\sigma$.  Finally,
\[
\Canon_3^k(\Lambda)
 =\bigsqcup_{\sigma\in\Sym_3}\Canon_3^{k,\sigma}(\Lambda),
\]
which gives \eqref{eq:global-count-reduction}.
\end{proof}

\begin{corollary}\label{cor:orbit-invariance}
If $\Lambda'=\tau\Lambda$ for some $\tau\in\Sym_3$, then
$c_3^k(\Lambda')=c_3^k(\Lambda)$ for every $k$.
\end{corollary}

\begin{proof}
Using \eqref{eq:global-count-reduction} and reindexing the sum by
$\rho=\tau^{-1}\sigma$ gives
\[
 c_3^k(\tau\Lambda)
 =\sum_{\sigma\in\Sym_3}L_k(\sigma^{-1}\tau\Lambda)
 =\sum_{\rho\in\Sym_3}L_k(\rho^{-1}\Lambda)
 =c_3^k(\Lambda).
\]
\end{proof}

\subsection{The trivial bijections and descent enumeration}

For a word $w=w_1\cdots w_N$ on $[n]$, let
\[
w^{\rev}=w_N\cdots w_1,
\qquad
w^{\comp}=(n+1-w_1)\cdots(n+1-w_N),
\qquad
w^{\rc}=(w^{\rev})^{\comp}.
\]
These operations preserve canon permutations.  Reverse-complement preserves
the number of descents, while reverse or complement changes descents
into ascents and vice versa.

For a word $w=w_1\cdots w_N$, define
\[
\Des(w)=\{i\in[N-1]:w_i>w_{i+1}\},
\qquad
\des(w)=\card{\Des(w)}.
\]
For $\Lambda\subseteq\Sym_3$, its canon descent polynomial is
\begin{equation}\label{eq:descent-polynomial}
D_{\Lambda,k}(x)
 =\sum_{\pi\in\Canon_3^k(\Lambda)}x^{\des(\pi)}.
\end{equation}

A nonzero polynomial $f(x)=\sum_{j=0}^d a_jx^j$ is
\emph{palindromic of degree $d$} if $a_j=a_{d-j}$ for all $j$.  It has a
unique expansion
\begin{equation}\label{eq:gamma-expansion}
f(x)=\sum_{i=0}^{\lfloor d/2\rfloor}
\gamma_i x^i(1+x)^{d-2i}.
\end{equation}
It is \emph{$\gamma$-positive} if every $\gamma_i\geq0$.  We adopt the
convention that the zero polynomial is palindromic of every degree and is
$\gamma$-positive.

\subsection{Dyck paths, Fuss--Dyck paths, and Narayana polynomials}

A Dyck path of semilength $k$ has steps $U=(1,1)$ and $D=(1,-1)$, starts and
ends at height zero, and never goes below the axis.  Let $\Dyck_k$ denote the set of such paths.  Binary lattice words with $k$ copies of a smaller and a larger letter are identified with $\Dyck_k$ by sending the smaller letter to
$U$ and the larger to $D$. It is clear that the lattice condition makes the path stay weakly above the $x$-axis.  A valley in a Dyck path is an occurrence $DU$, and
\[
N_k(x)=\sum_{P\in\Dyck_k}x^{\val(P)}
 =\sum_{j=0}^{k-1}\frac1k\binom{k}{j}\binom{k}{j+1}x^j
\]
is the $k$-th Narayana polynomial; we also set $N_0(x)=1$.  Its
$\gamma$-expansion is
\begin{equation}\label{eq:narayana-gamma}
N_k(x)=
\sum_{i=0}^{\lfloor(k-1)/2\rfloor}
\binom{k-1}{2i}C_i\,x^i(1+x)^{k-1-2i}.
\end{equation}

For $k\geq2$, a $(k-1)$-Fuss--Dyck path with $n$ down-steps is a path with
$(k-1)n$ steps $U=(1,1)$ and $n$ steps $D=(1,-(k-1))$ that never goes below
height zero.  Let $\FussDyck_{k,n}$ denote this set.  Its cardinality is
\begin{equation}\label{eq:fuss-catalan}
\FC_{k,n}
 =\frac1{(k-1)n+1}\binom{kn}{n}
 =\frac1n\binom{kn}{n-1}.
\end{equation}
These paths are also in bijection with full ordered $k$-ary trees having $n$
internal vertices.

\subsection{The four three-letter counting sequences}

We use
\begin{equation}\label{eq:C-and-T}
C_k=\frac1{k+1}\binom{2k}{k},
\qquad
T_k=\card{\Lat_3^k}
 =f^{(k,k,k)}
 =\frac{2(3k)!}{k!(k+1)!(k+2)!}.
\end{equation}
The additional classes are
\begin{equation}\label{eq:B-and-Q-def}
B_k
 =\card{\{w\in\Lat_3^k:\lastpos_w(1)<\firstpos_w(3)\}},
\qquad
Q_k=\card{\Lat_3^k(321)}
\end{equation}
where $\lastpos_w(i), \firstpos_w(i)$ denote the positions of the last and first occurrence of the letter $i$ in $w$ respectively.
We set $C_0=T_0=B_0=Q_0=1$.  Their first values are
\begin{center}
\begin{tabular}{c|rrrrrrr}
\toprule
$k$ & $0$ & $1$ & $2$ & $3$ & $4$ & $5$ & $6$\\
\midrule
$C_k$ & $1$ & $1$ & $2$ & $5$ & $14$ & $42$ & $132$\\
$T_k$ & $1$ & $1$ & $5$ & $42$ & $462$ & $6006$ & $87516$\\
$B_k$ & $1$ & $1$ & $4$ & $21$ & $121$ & $728$ & $4488$\\
$Q_k$ & $1$ & $1$ & $5$ & $33$ & $234$ & $1706$ & $12618$\\
\bottomrule
\end{tabular}
\end{center}

\subsection{D-finite series}

A formal power series is D-finite if it satisfies a linear differential
equation with polynomial coefficients.  A sequence is P-recursive if it
satisfies a linear recurrence with polynomial coefficients.  A sequence is
P-recursive if and only if its ordinary generating function is D-finite.
We use the standard closure properties of holonomic sums
\cite{StanleyEC2,PetkovsekWilfZeilberger}.
\section{The complete three-letter structural classification}
\label{sec:three-letter-structure}

We now specialize to the alphabet \(\{1,2,3\}\).  We show that \(123\) is unavoidable, \(132\) and \(213\) force the terminal or initial
blocks into special forms, while \(231\) and \(312\) impose the same separation condition.

\begin{lemma}
\label{lem:123-unavoidable}
Every word in \(\Lat_3^k\) contains \(123\).
\end{lemma}

\begin{proof}
By Lemma~\ref{lem:ordered-occurrences}, the first occurrences satisfy
\[
p_{1,1}<p_{2,1}<p_{3,1}.
\]
This gives an occurrence of \(123\).
\end{proof}

Let \(\mathcal D_k\) be the set of binary lattice words with \(k\) copies of
each of two ordered letters. As noted earlier, identifying the smaller letter with \(U\) and
the larger with \(D\) gives the usual bijection
\(\mathcal D_k\simeq\Dyck_k\).

\begin{theorem}
\label{thm:block-specifications}
For every \(k\geq1\), the following descriptions hold.
\begin{enumerate}
\item A word \(w\in\Lat_3^k\) avoids \(132\) if and only if
\[
w=v3^k,
\]
where \(v\) is a binary lattice word on \(\{1,2\}\).

\item A word \(w\in\Lat_3^k\) avoids \(213\) if and only if
\[
w=1^kv,
\]
where \(v\) is a binary lattice word on \(\{2,3\}\).

\item The simultaneous avoidance class is the singleton
\[
\Lat_3^k(132,213)=\{1^k2^k3^k\}.
\]

\item For \(w\in\Lat_3^k\), the following are equivalent:
\begin{enumerate}
\item \(w\) contains \(231\);
\item \(w\) contains \(312\);
\item an occurrence of \(3\) precedes a later occurrence of \(1\).
\end{enumerate}
Consequently,
\begin{equation}\label{eq:separator-class}
\Lat_3^k(231)=\Lat_3^k(312)
 =\{w\in\Lat_3^k:p_{1,k}<p_{3,1}\},
\end{equation}
and adding \(321\) to either forbidden set does not change the class.
\end{enumerate}
\end{theorem}

\begin{proof}
Suppose that \(w\) avoids \(132\).  If an occurrence of \(3\) is not in
the final block, then it is followed by a \(1\) or a \(2\).  The prefix ending
at that \(3\) already contains a \(1\).  If a later \(2\) exists, these three
letters form \(1,3,2\).  If instead the selected \(3\) is followed by a \(1\),
then not all \(2\)'s can have occurred before it: otherwise the prefix
inequality \(\#_1\geq\#_2=k\) would imply that all \(1\)'s had also occurred.
Thus a later \(2\) again exists, and we obtain \(132\).  Hence all \(3\)'s form
the terminal block \(3^k\).  Deleting it leaves a binary lattice word on
\(\{1,2\}\).  Conversely, in a word \(v3^k\), no \(2\) follows a \(3\), so
\(132\) cannot occur.  This proves (i).

Part (ii) follows from the reverse-complementation of (i).  If both conditions hold, the word
has an initial block \(1^k\) and a terminal block \(3^k\). The remaining
letters are all \(2\)'s, giving (iii).

For (iv), either a \(231\)- or a \(312\)-occurrence has a \(3\) before
a later \(1\).  Conversely, suppose a \(3\) precedes a later \(1\).  The prefix
ending at the selected \(3\) contains a \(2\), so choosing such a \(2\), then
the \(3\), then the later \(1\), gives \(231\).  Moreover, not all \(2\)'s can
occur before the selected \(3\), for that would force all \(1\)'s to have
occurred there as well.  A later \(2\) therefore exists, and \(3,1,2\) gives
\(312\).  Thus avoidance of either pattern is exactly the condition that all
\(1\)'s precede all \(3\)'s.  Under this condition a decreasing triple
\(3,2,1\) is impossible, so forbidding \(321\) does nothing.
\end{proof}

\begin{example}
For \(k=3\), the word
\[
112122333
\]
avoids \(132\) and is obtained from the Dyck word \(112122\) by appending
\(333\).  Its reverse-complement \(111223233\) illustrates the \(213\)-class.
The word
\[
112212333
\]
has all \(1\)'s before all \(3\)'s and hence avoids both \(231\) and \(312\),
but it need not have a forced initial or terminal Catalan block.
\end{example}
These constraints give the complete six-class classification.
\begin{theorem}
\label{thm:six-state}
Let \(\Omega\subseteq\Sym_3\).  Exactly one of the following cases applies.
\begin{enumerate}
\item If \(123\in\Omega\), then \(L_k(\Omega)=0\).

\item If \(123\notin\Omega\) and \(132,213\in\Omega\), then
\[
L_k(\Omega)=1,
\]
and the unique word is \(1^k2^k3^k\).

\item If \(123\notin\Omega\) and exactly one of \(132,213\) lies in
\(\Omega\), then
\[
L_k(\Omega)=C_k.
\]
The class is a binary lattice word with a forced initial or terminal block.

\item If none of \(123,132,213\) lies in \(\Omega\), and at least one of
\(231,312\) lies in \(\Omega\), then
\[
L_k(\Omega)=B_k.
\]
Equivalently, every \(1\) precedes every \(3\).

\item If \(\Omega=\{321\}\), then \(L_k(\Omega)=Q_k\).

\item If \(\Omega=\varnothing\), then \(L_k(\Omega)=T_k\).
\end{enumerate}
Thus every fixed-underlying-permutation avoidance component has one of exactly six
possible cardinalities.
\end{theorem}

\begin{proof}
The first case is Lemma~\ref{lem:123-unavoidable}.  In the second case,
Theorem~\ref{thm:block-specifications}(iii) leaves the unique word
\(1^k2^k3^k\).  That word contains no length-three permutation pattern other
than \(123\), which is not forbidden, so any additional members of \(\Omega\)
do not remove it.

In the third case, Theorem~\ref{thm:block-specifications}(i) or (ii) leaves an
arbitrary binary lattice word, counted by \(C_k\).  In the \(132\)-avoiding
form \(v3^k\), the only possible length-three patterns are \(123\) and
\(213\); in the \(213\)-avoiding form \(1^kv\), they are \(123\) and
\(132\). Hence the remaining possible prohibitions have no effect.

In the fourth case, Theorem~\ref{thm:block-specifications}(iv) gives exactly
the class defining \(B_k\); the possible additional prohibition
\(321\) is redundant.  Once the first four mutually exclusive cases are
excluded, the only patterns that may remain forbidden are \(321\) alone, or
none at all.  These are the definitions of \(Q_k\) and \(T_k\), respectively.
\end{proof}

\section{The twelve enumerative types}
\label{sec:complete-enumeration}

Combining the decomposition of Canon permutations into separate underlying permutation sets with the six-class theorem of the previous section gives the following theorem.

\begin{theorem}[Twelve-class enumeration]
\label{thm:twelve-class-enumeration}
For every \(k\geq1\), the value of \(c_3^k(\Lambda)\) is given by the following
classification.

\medskip
\noindent\textbf{No forbidden patterns.}
\[
c_3^k(\varnothing)=6T_k.
\]

\medskip
\noindent\textbf{One forbidden pattern.}
For every \(\tau\in\Sym_3\),
\[
c_3^k(\tau)=Q_k+2B_k+2C_k.
\]

\medskip
\noindent\textbf{Two forbidden patterns.}
\begin{center}
\small
\begin{tabularx}{\textwidth}{>{\raggedright\arraybackslash}p{0.19\textwidth}X}
\toprule
Formula & Pattern sets \(\Lambda\)\\
\midrule
\(4C_k\)
& \(\{132,231\}\), \(\{213,312\}\), \(\{123,321\}\)\\[2pt]
\(B_k+2C_k+1\)
& \(\{132,213\}\), \(\{123,231\}\), \(\{123,312\}\),
  \(\{231,312\}\), \(\{132,321\}\), \(\{213,321\}\)\\[2pt]
\(2B_k+2C_k\)
& \(\{123,132\}\), \(\{123,213\}\), \(\{213,231\}\),
  \(\{132,312\}\), \(\{231,321\}\), \(\{312,321\}\)\\
\bottomrule
\end{tabularx}
\end{center}

\medskip
\noindent\textbf{Three forbidden patterns.}
\begin{center}
\small
\begin{tabularx}{\textwidth}{>{\raggedright\arraybackslash}p{0.19\textwidth}X}
\toprule
Formula & Pattern sets \(\Lambda\)\\
\midrule
\(3\)
& \(\{123,231,312\}\), \(\{132,213,321\}\)\\[2pt]
\(B_k+2C_k\)
& \(\{123,132,213\}\), \(\{123,213,231\}\),
  \(\{123,132,312\}\), \(\{213,231,321\}\),
  \(\{132,312,321\}\), \(\{231,312,321\}\)\\[2pt]
\(2C_k+1\)
& \(\{123,132,231\}\), \(\{132,213,231\}\),
  \(\{123,213,312\}\), \(\{132,213,312\}\),
  \(\{132,231,312\}\), \(\{213,231,312\}\),
  \(\{123,132,321\}\), \(\{123,213,321\}\),
  \(\{123,231,321\}\), \(\{132,231,321\}\),
  \(\{123,312,321\}\), \(\{213,312,321\}\)\\
\bottomrule
\end{tabularx}
\end{center}

\medskip
\noindent\textbf{Four forbidden patterns.}
\begin{center}
\small
\begin{tabularx}{\textwidth}{>{\raggedright\arraybackslash}p{0.19\textwidth}X}
\toprule
Formula & Pattern sets \(\Lambda\)\\
\midrule
\(2C_k\)
& \(\{123,132,213,231\}\), \(\{123,132,213,312\}\),
  \(\{123,213,231,321\}\), \(\{123,132,312,321\}\),
  \(\{132,231,312,321\}\), \(\{213,231,312,321\}\)\\[2pt]
\(2\)
& \(\{123,132,231,312\}\), \(\{123,213,231,312\}\),
  \(\{132,213,231,312\}\), \(\{123,132,213,321\}\),
  \(\{123,132,231,321\}\), \(\{132,213,231,321\}\),
  \(\{123,213,312,321\}\), \(\{132,213,312,321\}\),
  \(\{123,231,312,321\}\)\\
\bottomrule
\end{tabularx}
\end{center}

\medskip
\noindent\textbf{Five or six forbidden patterns.}
If \(\card{\Lambda}=5\), then \(c_3^k(\Lambda)=1\).  If
\(\Lambda=\Sym_3\), then \(c_3^k(\Lambda)=0\).
\end{theorem}

\begin{proof}
For each \(\Lambda\), Theorem~\ref{thm:fixed-underlying-reduction} expresses
the count as the sum of the six classes
\(L_k(\sigma^{-1}\Lambda)\).  Applying
Theorem~\ref{thm:six-state} gives the following complete list of components.
\begin{center}
\small
\begin{tabular}{c|c}
\toprule
Total formula & multiset of six component classes\\
\midrule
\(6T_k\) & \(T_k,T_k,T_k,T_k,T_k,T_k\)\\
\(Q_k+2B_k+2C_k\) & \(Q_k,B_k,B_k,C_k,C_k,0\)\\
\(4C_k\) & \(C_k,C_k,C_k,C_k,0,0\)\\
\(B_k+2C_k+1\) & \(B_k,C_k,C_k,1,0,0\)\\
\(2B_k+2C_k\) & \(B_k,B_k,C_k,C_k,0,0\)\\
\(3\) & \(1,1,1,0,0,0\)\\
\(B_k+2C_k\) & \(B_k,C_k,C_k,0,0,0\)\\
\(2C_k+1\) & \(C_k,C_k,1,0,0,0\)\\
\(2C_k\) & \(C_k,C_k,0,0,0,0\)\\
\(2\) & \(1,1,0,0,0,0\)\\
\(1\) & \(1,0,0,0,0,0\)\\
\(0\) & \(0,0,0,0,0,0\)\\
\bottomrule
\end{tabular}
\end{center}
A direct application of the six relabelings to one representative of each
left orbit produces the pattern lists in the statement. Summing each components gives the asserted formula.
\end{proof}
\begin{corollary}
\label{cor:four-catalan-count}
For every \(k\geq1\),
\[
c_3^k(213,312)=c_3^k(132,231)=c_3^k(123,321)=4C_k.
\]
More precisely, the underlying underlying permutations are
\begin{align*}
\{213,312\}:&\quad 123,132,231,321,\\
\{132,231\}:&\quad 123,213,312,321,\\
\{123,321\}:&\quad 132,213,231,312.
\end{align*}
Each underlying component is naturally bijective with \(\Dyck_k\) by deleting the
forced block in Theorem~\ref{thm:block-specifications}.
\end{corollary}

\begin{proof}
The count is the \(4C_k\) row of
Theorem~\ref{thm:twelve-class-enumeration}.  To identify the components, compute
\(\sigma^{-1}\Lambda\) for the six values of \(\sigma\).  In precisely the
four displayed cases, it contains exactly one of \(132,213\) and does not
contain \(123\); Theorem~\ref{thm:six-state}(iii) gives a Catalan component and
Theorem~\ref{thm:block-specifications} supplies the explicit Dyck bijection.
For the other two underlying permutations, \(123\in\sigma^{-1}\Lambda\), so
the component is empty.
\end{proof}

\begin{example}
For \(\Lambda=\{213,312\}\), choose the Dyck word \(v=223233\) on
\(\{2,3\}\), corresponding to \(UUDUDD\).  The first two components use the row
word \(111223233\), while the last two use the reverse-block form
\(112122333\) associated with the same unlabelled Dyck path.  Applying
\(\Phi_{123},\Phi_{132},\Phi_{231},\Phi_{321}\) produces one canon
permutation in each of the four Catalan pieces.
\end{example}
The twelve formulas also imply D-finiteness.
\begin{corollary}
\label{cor:all-D-finite}
For every \(\Lambda\subseteq\Sym_3\), the sequence
\(k\mapsto c_3^k(\Lambda)\) is P-recursive, and its ordinary generating
function is D-finite.  If its formula contains only \(B_k\), \(C_k\), and
constants, the ordinary generating function is algebraic.
\end{corollary}

\begin{proof}
The sequences \(C_k\) and \(T_k\) are hypergeometric and hence P-recursive.
The sequence \(B_k\) has the algebraic generating function proved in
Theorem~\ref{thm:B-algebraic}, while \(Q_k\) is P-recursive by
Corollary~\ref{cor:Q-holonomic}.  The class of D-finite series is closed under
finite linear combinations, so the twelve formulas give the first assertion.
The Catalan series and \(B(z)\) are algebraic, and algebraic series are closed
under finite sums and scalar multiplication, giving the second assertion.
\end{proof}

\section[Descents and gamma-positivity]{Descents and \(\gamma\)-positivity}
\label{sec:gamma}
The total enumeration does not determine the descent distribution, because the relabeling $\Phi_\sigma$ changes comparisons between adjacent row letters.  For the Catalan components, however, this change is a constant shift depending only on the underlying permutation.

For a fixed component, put
\[
D^{\sigma}_{\Lambda,k}(x)
 =\sum_{\pi\in\Canon_3^{k,\sigma}(\Lambda)}x^{\des(\pi)}.
\]

\begin{lemma}
\label{lem:catalan-descent-shift}
Let \(P\in\Dyck_k\), and let \(v\) be its binary lattice word.
If \(w=1^kv\), where \(v\) uses the row letters \(2,3\), or if
\(w=v3^k\), where \(v\) uses the row letters \(1,2\), then for every
\(\sigma\in\Sym_3\),
\begin{equation}\label{eq:catalan-descent-shift}
\des(\Phi_\sigma(w))=\val(P)+\des(\sigma).
\end{equation}
Consequently, every nonempty Catalan component has descent polynomial
\[
x^{\des(\sigma)}N_k(x).
\]
\end{lemma}

\begin{proof}
    Firstly, consider \(w=1^kv\).  The boundary between the constant block and the
binary word contributes a descent exactly when \(\sigma_1>\sigma_2\).  Within
\(v\), every transition from $3$ to $2$ is a valley and every transition $2$ to $3$ is a peak.  A nonempty Dyck path has one more peak than valley.  If
\(\sigma_2<\sigma_3\), the valleys become descents, so the total is
\[
\val(P)+\ind_{\sigma_1>\sigma_2}.
\]
If \(\sigma_2>\sigma_3\), only the peaks become descents, so the total is
\[
\val(P)+1+\ind_{\sigma_1>\sigma_2}.
\]
In both cases this is \(\val(P)+\des(\sigma)\).

For \(w=v3^k\), the same argument uses the transitions from $1$ to $2$ and
$2$ to $1$, together with the final boundary \(\sigma_2\sigma_3\).  If
\(\sigma_1<\sigma_2\), the valleys contribute descents and if
\(\sigma_1>\sigma_2\), the peaks contribute descents.  The result is once  again
\(\val(P)+\des(\sigma)\).  Summing over \(P\in\Dyck_k\) gives the required
polynomial.
\end{proof}
\begin{corollary}
\label{cor:four-catalan-descents}
For every \(k\geq1\),
\begin{align}
D_{\{213,312\},k}(x)
 &=D_{\{132,231\},k}(x)=(1+x)^2N_k(x),
 \label{eq:middle-extremum-descent}\\
D_{\{123,321\},k}(x)&=4xN_k(x).
 \label{eq:monotone-pair-descent}
\end{align}
\end{corollary}

\begin{proof}
For \(\{213,312\}\), Corollary~\ref{cor:four-catalan-count} gives the
underlying permutations
\(123,132,231,321\), whose descent numbers are \(0,1,1,2\).  The preceding
lemma therefore gives
\[
(1+2x+x^2)N_k(x)=(1+x)^2N_k(x).
\]
The class \(\{132,231\}\) has underlying permutations
\(123,213,312,321\), with the same descent-number multiset.  For
\(\{123,321\}\), all four underlying permutations
\(132,213,231,312\) have one descent, giving \(4xN_k(x)\).
\end{proof}
\begin{corollary}
\label{cor:four-catalan-gamma}
The avoidance classes of $\{213,312\},\{132,231\}$ have the expansion
\[
(1+x)^2N_k(x)
 =\sum_{i=0}^{\lfloor(k-1)/2\rfloor}
 \binom{k-1}{2i}C_i\,x^i(1+x)^{k+1-2i}.
\]
Thus their \(i\)-th \(\gamma\)-coefficient is
\[
\gamma_{k,i}=\binom{k-1}{2i}C_i.
\]
\end{corollary}

\begin{proof}
Multiply the Narayana expansion \eqref{eq:narayana-gamma} by \((1+x)^2\).
\end{proof}

\begin{proposition}
\label{prop:explicit-positive-formulas}
For every \(k\geq1\), the following identities hold:
\begin{center}
\small
\begin{tabularx}{\textwidth}{>{\raggedright\arraybackslash}p{0.56\textwidth}X}
\toprule
Forbidden set \(\Lambda\) & \(D_{\Lambda,k}(x)\)\\
\midrule
\(\{132,231\}\), \(\{213,312\}\)
& \((1+x)^2N_k(x)\)\\
\(\{132,231,312,321\}\), \(\{213,231,312,321\}\)
& \((1+x)N_k(x)\)\\
\(\{132,213,231,321\}\), \(\{132,213,312,321\}\)
& \(1+x\)\\
\(\Sym_3\setminus\{123\}\)
& \(1\)\\
\(\Sym_3\)
& \(0\)\\
\bottomrule
\end{tabularx}
\end{center}
The unrestricted polynomial \(D_{\varnothing,k}(x)\) is \(\gamma\)-positive
by the canon-poset factorization of Beck and Deligeorgaki
\cite{BeckDeligeorgaki}.
\end{proposition}

\begin{proof}
The first row is Corollary~\ref{cor:four-catalan-descents}.  For
\(\Lambda=\{132,231,312,321\}\), the only nonempty components have underlying
permutations \(123\) and \(213\); both are Catalan components, with shifts zero
and one.  For \(\Lambda=\{213,231,312,321\}\), the corresponding underlying
permutations are \(123\) and \(132\).  Lemma~\ref{lem:catalan-descent-shift}
therefore gives \((1+x)N_k(x)\) in both cases.

For \(\{132,213,231,321\}\), the two underlying components are the singleton word
\(1^k2^k3^k\) relabeled by \(123\) and \(312\); for
\(\{132,213,312,321\}\), they are relabeled by \(123\) and \(231\).
The descent shifts are zero and one, giving \(1+x\).  If all patterns except
\(123\) are forbidden, only the increasing underlying permutation survives,
and if every pattern is forbidden, no component survives.
\end{proof}

At
\(k=1\), a canon permutation is simply an element of \(\Sym_3\), and hence
\begin{equation}\label{eq:k1-polynomial}
D_{\Lambda,1}(x)=\sum_{\sigma\in\Sym_3\setminus\Lambda}x^{\des(\sigma)}.
\end{equation}
There are only five three-row lattice words at \(k=2\):
\begin{equation}\label{eq:five-lattice-words}
112233,\quad112323,\quad121233,\quad121323,\quad123123.
\end{equation}
Relabeling these by the six underlying permutations gives a direct exhaustive
calculation of every \(D_{\Lambda,2}(x)\).

\begin{lemma}
\label{lem:two-level-census}
Exactly seventeen forbidden sets have a \(\gamma\)-positive polynomial at
\(k=1\).  Their \(k=2\) polynomials are listed below; the rows marked
``True'' are precisely the eight restricted sets in
Proposition~\ref{prop:explicit-positive-formulas}, together with the
unrestricted set.
\begin{center}
\scriptsize
\setlength{\tabcolsep}{3pt}
\begin{tabularx}{0.99\textwidth}{>{\raggedright\arraybackslash}X
 >{\raggedright\arraybackslash}p{0.30\textwidth}c}
\toprule
\(\Lambda\) & \(D_{\Lambda,2}(x)\) & \(\gamma\)-positive?\\
\midrule
\(\varnothing\) & \(1+7x+14x^2+7x^3+x^4\) & True\\
\(\{132\}\) or \(\{213\}\) & \(1+4x+8x^2+4x^3\) & False\\
\(\{231\}\) or \(\{312\}\) & \(1+5x+8x^2+3x^3\) & False\\
\(\{132,213\}\) & \(1+2x+3x^2+3x^3\) & False\\
\(\{132,231\}\) or \(\{213,312\}\)
 & \(1+3x+3x^2+x^3\) & True\\
\(\{213,231\}\) or \(\{132,312\}\)
 & \(1+3x+6x^2+2x^3\) & False\\
\(\{231,312\}\) & \(1+4x+4x^2\) & False\\
\(\{132,213,231,321\}\) or \(\{132,213,312,321\}\)
 & \(1+x\) & True\\
\(\{132,231,312,321\}\) or \(\{213,231,312,321\}\)
 & \(1+2x+x^2\) & True\\
\(\Sym_3\setminus\{123\}\) & \(1\) & True\\
\(\Sym_3\) & \(0\) & True\\
\bottomrule
\end{tabularx}
\end{center}
There are five additional classes that are palindromic at \(k=1\).  At
\(k=2\),
\begin{align*}
D_{\{132,213,231\},2}(x)
 =D_{\{132,213,312\},2}(x)&=1+x+2x^2+x^3,\\
D_{\{132,231,312\},2}(x)
 =D_{\{213,231,312\},2}(x)&=1+2x+2x^2,\\
D_{\{132,213,231,312\},2}(x)&=1+x^2.
\end{align*}
Thus only the final class remains palindromic, and its \(\gamma\)-vector is
\((1,-2)\).
\end{lemma}

\begin{proof}
Equation~\eqref{eq:k1-polynomial} reduces the first level to the descent
multiset
\[
\des(123)=0,\qquad
\des(132)=\des(213)=\des(231)=\des(312)=1,\qquad
\des(321)=2.
\]
Checking the palindromicity and \(\gamma\)-coefficients of the resulting
quadratic, linear, constant, or zero polynomials leaves exactly the seventeen
sets displayed in the first table and the five additional palindromic sets.

For \(k=2\), every canon permutation is obtained uniquely by applying one of
the six maps \(\Phi_\sigma\) to one of the five words in
\eqref{eq:five-lattice-words}.  Testing all subsequences that use one copy of each row letter and recording
adjacent descents gives the
polynomials displayed above. The entries of the last column follow by symmetry of coefficients and the unique \(\gamma\)-expansion.
\end{proof}

\begin{theorem}
\label{thm:uniform-gamma-iff}
For \(\Lambda\subseteq\Sym_3\), the following are equivalent:
\begin{enumerate}
\item \(D_{\Lambda,k}(x)\) is \(\gamma\)-positive for every \(k\geq1\);
\item both \(D_{\Lambda,1}(x)\) and \(D_{\Lambda,2}(x)\) are
\(\gamma\)-positive;
\item
\[
\begin{split}
\Lambda\in\bigl\{&\varnothing,
\{132,231\},\{213,312\},
\{132,213,231,321\},\{132,213,312,321\},\\
&\{132,231,312,321\},\{213,231,312,321\},
\Sym_3\setminus\{123\},\Sym_3\bigr\}.
\end{split}
\]
\end{enumerate}
\end{theorem}

\begin{proof}
The implication (i)\(\Rightarrow\)(ii) is immediate.  By
Lemma~\ref{lem:two-level-census}, condition (ii) leaves exactly the nine sets
in (iii).  Proposition~\ref{prop:explicit-positive-formulas}, the Narayana
expansion \eqref{eq:narayana-gamma}, and the unrestricted result of Beck and
Deligeorgaki show that every set in (iii) is \(\gamma\)-positive for all
\(k\), proving (iii)\(\Rightarrow\)(i).
\end{proof}

\begin{theorem}
\label{thm:uniform-palindromicity}
The polynomial \(D_{\Lambda,k}(x)\) is palindromic for every \(k\geq1\) if
and only if \(\Lambda\) is one of the nine sets in
Theorem~\ref{thm:uniform-gamma-iff}, or
\[
\Lambda=\{132,213,231,312\}.
\]
For this unique additional class,
\[
D_{\Lambda,k}(x)=1+x^2=(1+x)^2-2x
\]
for every \(k\), so it is not \(\gamma\)-positive.
Moreover, uniform palindromicity is also witnessed at levels \(k=1,2\).
\end{theorem}

\begin{proof}
Lemma~\ref{lem:two-level-census} shows that the only class palindromic at both
levels but absent from the \(\gamma\)-positive list is
\(\Lambda=\{132,213,231,312\}\).  For this class, the only nonempty components
are singleton components with underlying permutations \(123\) and \(321\).  The
unique row word is \(1^k2^k3^k\), and relabeling it gives descent numbers zero
and two.  Hence the polynomial is \(1+x^2\) for every \(k\).  All other
classes fail palindromicity at level one or two, while the nine classes in
Theorem~\ref{thm:uniform-gamma-iff} are palindromic because they are
\(\gamma\)-positive.
\end{proof}

We finish the section with a Boolean-orbit interpretation of the
$\gamma$-coefficients.  Let $\mathcal M_n^{(2)}$ be the set of Motzkin paths of
length $n$ whose level steps have colours $H_1$ and $H_2$.  The following
recursive map is a bijection
\[
f_n:\mathcal M_n^{(2)}\longrightarrow\Dyck_{n+1}.
\]
Set $f_0(\varnothing)=UD$.  For $n\geq1$, use the first applicable case:
\begin{align*}
f_n(H_1M')&=U f_{n-1}(M')D,\\
f_n(H_2M')&=UD f_{n-1}(M'),\\
f_n(UM_1DM_2)&=U f_i(M_1)D f_{n-2-i}(M_2),
\end{align*}
where, in the last line, $UM_1D$ is the first-return block of the Motzkin path
and $M_1\in\mathcal M_i^{(2)}$.  The inverse is obtained from the first-return
decomposition of a nonempty Dyck path: a primitive path gives the first case,
a path beginning with $UD$ gives the second case, and all remaining paths give
the third case.  Hence a Dyck path of semilength $k$ corresponds to a
two-coloured Motzkin path of length $k-1$.

If $P=f_{k-1}(M)$, then
\begin{equation}\label{eq:motzkin-valley-weight}
\val(P)=u(M)+b(M),
\end{equation}
where $u(M)$ is the number of up-steps and $b(M)$ is the number of level steps
of colour $H_2$.  Indeed, the $H_1$ case creates no new valley, while the
$H_2$ case and the first-return case each create exactly one.  Forgetting the
level-step colours leaves a Motzkin skeleton with $i$ up-steps and
$k-1-2i$ level steps.

\begin{theorem}
\label{thm:motzkin-toggle-action}
Every restricted class in Theorem~\ref{thm:uniform-gamma-iff} admits a Boolean toggle action of an elementary abelian $2$-group whose orbit enumerators are the terms in its $\gamma$-expansion.
For $\Lambda=\{213,312\}$, there is a bijection
\[
\Canon_3^k(\Lambda)\longleftrightarrow\Dyck_k\times\{0,1\}^2
\]
under which
\[
\des(\pi)=\val(P)+\varepsilon_1+\varepsilon_2.
\]
The two external bits encode
\[
(0,0)\leftrightarrow123,
\quad(1,0)\leftrightarrow132,
\quad(0,1)\leftrightarrow231,
\quad(1,1)\leftrightarrow321.
\]
Via the two-coloured Motzkin encoding, the group
$\left(\mathbb Z/2\mathbb Z\right)^{k+1}$ acts by toggling the colour of a level step at
each of the $k-1$ positions where such a step is present, fixing a path when
the selected position is not level, and by toggling the two external bits.
If the uncoloured Motzkin path has $i$ up-steps, its orbit enumerator is
\[
x^i(1+x)^{k+1-2i}.
\]
The $i$-th $\gamma$-coefficient is the number of uncolored Motzkin paths of length $k-1$ with 
$i$ up steps. The reverse-complement class $\{132,231\}$ has the same action.  The classes
with polynomial $(1+x)N_k(x)$ use one external toggle; the classes with
polynomial $1+x$ use one external toggle and no path component.  The constant
and empty classes carry the trivial action.
\end{theorem}

\begin{proof}
The four component bijections in Corollary~\ref{cor:four-catalan-count}
identify an object with a Dyck path and one of four underlying permutations.
The bit labeling keeps track of the descent shifts $0,1,1,2$, so
Lemma~\ref{lem:catalan-descent-shift} gives the displayed statistic formula.
Equation~\eqref{eq:motzkin-valley-weight} shows that an uncoloured skeleton with
$i$ up-steps contributes the mandatory factor $x^i$.  Each of its
$k-1-2i$ level steps can be toggled between weights $1$ and $x$, and the two
external bits can be toggled independently.  Hence the orbit enumerator is
\[
x^i(1+x)^{k-1-2i}(1+x)^2=x^i(1+x)^{k+1-2i}.
\]
The same argument with one or zero external bits gives the remaining formulas.
This action is analogous in spirit to the orbit decompositions in
Foata--Strehl theory; see, for example, \cite{Branden}.
\end{proof}

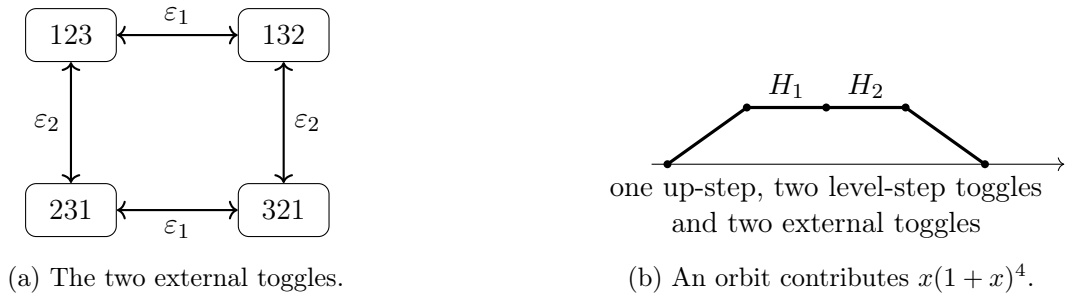
\begin{figure}[H]
\centering
\begin{subfigure}{0.43\textwidth}
\centering
\begin{tikzpicture}[node distance=1.6cm,
 b/.style={draw,rounded corners,minimum width=1.2cm,minimum height=7mm}]
\node[b] (a) {123};
\node[b,right=of a] (b) {132};
\node[b,below=of a] (c) {231};
\node[b,right=of c] (d) {321};
\draw[<->,thick] (a)--node[above]{\(\varepsilon_1\)}(b);
\draw[<->,thick] (c)--node[below]{\(\varepsilon_1\)}(d);
\draw[<->,thick] (a)--node[left]{\(\varepsilon_2\)}(c);
\draw[<->,thick] (b)--node[right]{\(\varepsilon_2\)}(d);
\end{tikzpicture}
\caption{The two external toggles.}
\end{subfigure}
\hfill
\begin{subfigure}{0.52\textwidth}
\centering
\begin{tikzpicture}[x=1.05cm,y=0.75cm]
\draw[->] (-0.2,0)--(5.0,0);
\draw[very thick] (0,0)--(1,1)--(2,1)--(3,1)--(4,0);
\fill (0,0) circle (1.5pt);\fill (1,1) circle (1.5pt);
\fill (2,1) circle (1.5pt);\fill (3,1) circle (1.5pt);
\fill (4,0) circle (1.5pt);
\node at (1.5,1.38) {\(H_1\)};
\node at (2.5,1.38) {\(H_2\)};
\node[align=center] at (2,-0.75)
 {one up-step, two level-step toggles\\and two external toggles};
\end{tikzpicture}
\caption{An orbit contributes \(x(1+x)^4\).}
\end{subfigure}
\caption{The Boolean-orbit model for a four-Catalan class at \(k=5\).}
\label{fig:motzkin-toggle}
\end{figure}

\section[The sequence B]{The sequence \(B_k\)}
\label{sec:Bstate}

By Theorem~\ref{thm:block-specifications}, \(B_k\) counts the three-row
lattice words in which every \(1\) precedes every \(3\).  Cutting immediately
after the final \(1\) turns the separation condition into two independent
ballot problems.

\begin{lemma}[A shifted Vandermonde identity]
\label{lem:shifted-vandermonde}
For nonnegative integers \(r,s,m\) with \(m\leq s\), and with the
convention that an out-of-range binomial coefficient is zero,
\begin{equation}\label{eq:shifted-vandermonde}
\sum_{b\in\Z}
 \binom{r+b}{b}\binom{s-b}{m-b}
 =\binom{r+s+1}{m}.
\end{equation}
\end{lemma}

\begin{proof}
Use \(\binom{r+b}{b}=[u^b](1-u)^{-r-1}\) and sum after writing
\[
\binom{s-b}{m-b}=[x^{m-b}](1+x)^{s-b}.
\]
Equivalently, coefficient extraction and the negative-binomial series give
\begin{align*}
\sum_b\binom{r+b}{b}\binom{s-b}{m-b}
&=[x^m](1+x)^s
  \sum_{b\geq0}\binom{r+b}{b}\left(\frac{x}{1+x}\right)^b\\
&=[x^m](1+x)^s
  \left(1-\frac{x}{1+x}\right)^{-r-1}\\
&=[x^m](1+x)^{r+s+1}
 =\binom{r+s+1}{m}.
\end{align*}
\end{proof}
\begin{theorem}
\label{thm:B-closed-form}
For every \(k\geq0\),
\begin{equation}\label{eq:B-closed-form}
B_k=\binom{3k}{k}-2\binom{3k}{k-1}+\binom{3k}{k-2}
 =\frac{k^2+k+2}{2(k+1)(2k+1)}\binom{3k}{k}.
\end{equation}
Equivalently,
\begin{equation}\label{eq:B-ratio}
\frac{B_{k+1}}{B_k}
 =\frac{3(3k+1)(3k+2)(k^2+3k+4)}
 {2(k+2)(2k+3)(k^2+k+2)}.
\end{equation}
\end{theorem}

\begin{proof}
The cases $k=0$ and $k=1$ are immediate, so assume $k\geq2$.
For \(w\in\Lat_3^k(231)\), let \(b\) be the number of \(2\)'s before the
final \(1\).  We have \(0\leq b\leq k-1\), if all \(k\) copies of \(2\)
preceded the final \(1\), the prefix inequality would require all \(k\)
copies of \(1\) to have appeared earlier.

Delete the final \(1\) from the initial block.  A ballot word
with \(k-1\) copies of \(1\) and \(b\) copies of \(2\), counted by the
reflection principle as
\(
A_{k,b}=\binom{k+b-1}{b}-\binom{k+b-1}{b-1}
\) remains.
After the final \(1\), there are \(k-b\) copies of \(2\) and \(k\) copies of
\(3\).  At the cut the current excess of \(2\)'s over \(3\)'s is \(b\), and
the suffix must never make this excess negative. A second reflection gives
\[
S_{k,b}=\binom{2k-b}{k}-\binom{2k-b}{k+1}.
\]
The prefix and suffix choices are independent, so
\begin{equation}\label{eq:B-cut-sum}
B_k=\sum_{b=0}^{k-1}A_{k,b}S_{k,b}.
\end{equation}
The summand vanishes at \(b=k\), so the sum may be extended to all integers.
Expanding the product and applying
Lemma~\ref{lem:shifted-vandermonde} four times gives
\begin{align*}
\sum_b\binom{k+b-1}{b}\binom{2k-b}{k}
 &=\binom{3k}{k},\\
\sum_b\binom{k+b-1}{b}\binom{2k-b}{k+1}
 &=\binom{3k}{k-1},\\
\sum_b\binom{k+b-1}{b-1}\binom{2k-b}{k}
 &=\binom{3k}{k-1},\\
\sum_b\binom{k+b-1}{b-1}\binom{2k-b}{k+1}
 &=\binom{3k}{k-2}.
\end{align*}
This proves the first expression in \eqref{eq:B-closed-form}.  Dividing the
three binomial coefficients by \(\binom{3k}{k}\) and simplifying gives the
second.  Finally, taking the quotient of the closed forms at \(k+1\) and
\(k\) gives \eqref{eq:B-ratio}.
\end{proof}

\begin{figure}[H]
\centering
\begin{tikzpicture}[x=0.68cm,y=0.72cm]
\node at (-1.5,0) {\(w=\)};
\foreach \x/\a in {0/1,1/1,2/2,3/1,4/2,5/1,6/2,7/2,8/3,9/3,10/3,11/3}
  {\node[draw,minimum size=6mm] at (\x,0){\a};}
\draw[very thick,red!70!black] (5.5,-0.55)--(5.5,0.55);
\node[above] at (5.5,0.6) {final \(1\)};
\draw[decorate,decoration={brace,mirror,amplitude=5pt}]
 (-0.35,-0.6)--(5.35,-0.6)
 node[midway,below=7pt]{\(A_{k,b}\)};
\draw[decorate,decoration={brace,mirror,amplitude=5pt}]
 (5.65,-0.6)--(11.35,-0.6)
 node[midway,below=7pt]{\(S_{k,b}\)};
\end{tikzpicture}
\caption{The final-\(1\) cut used to enumerate the sequence.  The left
brace is the \(1/2\)-ballot prefix counted by \(A_{k,b}\), and the right brace
is the shifted \(2/3\)-ballot suffix counted by \(S_{k,b}\).  The pictured
word has \(k=4\) and \(b=2\).}
\label{fig:B-cut}
\end{figure}
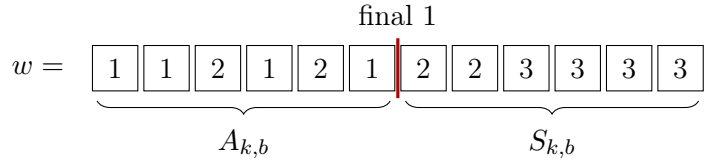

\begin{theorem}
\label{thm:B-algebraic}
Let \(B(z)=\sum_{k\geq0}B_kz^k\).  It is the unique formal power series with
constant term \(1\) satisfying
\begin{align}\label{eq:B-cubic}
0={}&(1-16z+64z^2)
 +(-1+21z-96z^2)B(z)\notag\\
&+(-4z+27z^2)B(z)^2
 +(-4z^2+27z^3)B(z)^3.
\end{align}
In particular, \(B(z)\) is algebraic of degree at most three over \(\Q(z)\).
\end{theorem}

\begin{proof}
Put
\[
t_k=\frac{1}{2k+1}\binom{3k}{k},
\qquad T(z)=\sum_{k\geq0}t_kz^k.
\]
The ternary-tree generating function satisfies
\begin{equation}\label{eq:ternary-tree}
T(z)=1+zT(z)^3.
\end{equation}
The closed formula in Theorem~\ref{thm:B-closed-form} can be rewritten as
\[
B_k=\left(\frac{k}{2}+\frac1{k+1}\right)t_k.
\]
Therefore
\begin{equation}\label{eq:B-in-T}
B(z)=\frac12zT'(z)+\frac1z\int_0^zT(u)\,du.
\end{equation}
Differentiating \eqref{eq:ternary-tree} and using
\(z=(T-1)/T^3\) gives
\[
zT'(z)=\frac{T(T-1)}{3-2T}.
\]
For the integral, change variables from \(u\) to the corresponding value
\(S=T(u)\).  Since
\(u=(S-1)/S^3\),
\[
\frac1z\int_0^zT(u)\,du
 =\frac1z\int_1^T\frac{3-2S}{S^3}\,dS
 =\frac{T(3-T)}{2}.
\]
Substitution into \eqref{eq:B-in-T} yields the rational parametrization
\begin{equation}\label{eq:B-parametrization}
B=\frac{T(T-2)^2}{3-2T},
\qquad z=\frac{T-1}{T^3}.
\end{equation}
Eliminating \(T\) from these two equations gives exactly
\eqref{eq:B-cubic}. At \(z=0\) the branch arising from
\(T(0)=1\) has \(B(0)=1\), and the formal implicit-function recursion
uniquely determines this branch coefficient by coefficient. This completes the proof.
\end{proof}

\begin{remark}
Borie and Falque proved that the sequence $B_k$ counts pairs \((T_1,T_2)\) of
plane binary trees of size $k$ for which the sum of the numbers of disabled
grafting sites is at most $k+1$; see \cite[Proposition~1]{BorieFalque}. 
\end{remark}

Define
\begin{equation}\label{eq:B-refinement}
\Bsf_k(x)=\sum_{w\in\Lat_3^k(231)}x^{\des(w)}.
\end{equation}
The first values are
\begin{align*}
\Bsf_1(x)&=1,\\
\Bsf_2(x)&=1+2x+x^2,\\
\Bsf_3(x)&=1+6x+10x^2+4x^3,\\
\Bsf_4(x)&=1+12x+43x^2+50x^3+15x^4.
\end{align*}

\begin{remark}[Real zeros in the tested range]
\label{obs:B-real-rooted}
We have verified that, for every
$1\leq k\leq25$, all zeros of $\Bsf_k(x)$ are real and negative.  Hence the
coefficient sequence is log-concave and unimodal throughout this tested range.
\end{remark}

\section[The sequence Q]{The sequence \(Q_k\)}
\label{sec:Qstate}

Under the standard labeling of the north-east rectangular poset, $Q_k$ is
the number of $123$-avoiding linear extensions of
$\operatorname{NE}_{3,k}$ studied by Anderson, Egge, Riehl, Ryan, Steinke, and Vaughan
\cite{AndersonEtAl}.  A cut at the final $1$ gives the following exact finite
multisum.

Use the convention \(\binom{m}{j}=0\) if \(j<0\) or \(j>m\).  For
\(b\geq1\) and \(r\geq b\), put
\begin{equation}\label{eq:E-rb}
E(r,b)=\binom{r+b-1}{b-1}-\binom{r+b-1}{b-2}.
\end{equation}
For \(0\leq c\leq b\leq k\), put
\begin{equation}\label{eq:H-kbc}
H_k(b,c)
 =\binom{2k-b-c}{k-b}-\binom{2k-b-c}{k-b-1}
 =\frac{b-c+1}{k-c+1}\binom{2k-b-c}{k-b}.
\end{equation}
\begin{theorem}
\label{thm:Q-multisum}
For every \(k\geq1\),
\begin{equation}\label{eq:Q-multisum}
Q_k=C_k+
\sum_{b=1}^{k-1}\sum_{c=0}^{b}
H_k(b,c)
\sum_{r=b}^{k-1}
E(r,b)\binom{k-1-r+c}{c}.
\end{equation}
Here \(b\) and \(c\) are the numbers of \(2\)'s and \(3\)'s before the final
\(1\), and \(r\) is the number of \(1\)'s before the last \(2\) in that
prefix.
\end{theorem}

\begin{proof}
Let \(w\in\Lat_3^k(321)\), and cut immediately after its final \(1\).  Let
\(b\) and \(c\) be the numbers of \(2\)'s and \(3\)'s in the prefix before
that final \(1\).  The lattice inequalities imply
\(0\leq c\leq b\leq k-1\).

If \(b=0\), then also \(c=0\); the word starts with \(1^k\), and its remaining
\(2/3\)-word is an arbitrary binary lattice word.  This gives the initial
term \(C_k\).

Assume \(b\geq1\).  Within the prefix, every \(2\) must precede every \(3\).
Indeed, if a \(3\) preceded a later \(2\), the final \(1\) would complete a
\(3,2,1\) occurrence.  Let \(r\) be the number of \(1\)'s occurring no later
than the last of those \(b\) copies of \(2\).  The initial block on
\(\{1,2\}\) has \(r\) ones and \(b\) twos, ends in \(2\), and is ballot.
Deleting its final \(2\) and applying the reflection principle gives
\[
E(r,b)=\binom{r+b-1}{b-1}-\binom{r+b-1}{b-2}.
\]
The range is \(b\leq r\leq k-1\): ballot condition gives \(r\geq b\), and the final
\(1\) has not yet appeared.

Between this last prefix \(2\) and the final \(1\), the remaining
\(k-1-r\) prefix copies of \(1\) may be shuffled freely with the \(c\) copies
of \(3\).  Such a shuffle cannot create \(321\), since no prefix \(2\) follows
a prefix \(3\), and it is counted by
\[
\binom{k-1-r+c}{c}.
\]
After the final \(1\), the suffix consists of \(k-b\) copies of \(2\) and
\(k-c\) copies of \(3\).  At the cut, the excess of twos over threes is
\(b-c\), and the suffix must preserve the inequality \(\#_2\geq\#_3\).
The shifted reflection principle gives
\[
H_k(b,c)
 =\binom{2k-b-c}{k-b}-\binom{2k-b-c}{k-b-1}.
\]
Once the final \(1\) has passed, no new \(321\) occurrence can be completed,
so these choices are independent and exhaustive.  Multiplying the three
factors and summing over \(r,c,b\) proves \eqref{eq:Q-multisum}.
\end{proof}

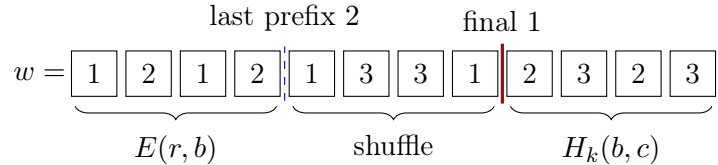
\begin{figure}[H]
\centering
\begin{tikzpicture}[x=0.72cm,y=0.72cm]
\node at (-1,0) {\(w=\)};
\foreach \x/\a in {0/1,1/2,2/1,3/2,4/1,5/3,6/3,7/1,8/2,9/3,10/2,11/3}
 {\node[draw,minimum size=6mm] at (\x,0){\a};}
\draw[very thick,red!70!black] (7.5,-0.55)--(7.5,0.55);
\draw[dashed,blue!70!black] (3.5,-0.5)--(3.5,0.5);
\node[above] at (7.5,0.6) {final \(1\)};
\node[above] at (3.5,0.6) {last prefix \(2\)};
\draw[decorate,decoration={brace,mirror,amplitude=5pt}]
 (-0.35,-0.6)--(3.35,-0.6)
 node[midway,below=7pt]{\(E(r,b)\)};
\draw[decorate,decoration={brace,mirror,amplitude=5pt}]
 (3.65,-0.6)--(7.35,-0.6)
 node[midway,below=7pt]{shuffle};
\draw[decorate,decoration={brace,mirror,amplitude=5pt}]
 (7.65,-0.6)--(11.35,-0.6)
 node[midway,below=7pt]{\(H_k(b,c)\)};
\end{tikzpicture}
\caption{The three factors in the last-one block decomposition.}
\label{fig:Q-cut}
\end{figure}
\begin{corollary}
\label{cor:Q-holonomic}
The sequence \((Q_k)_{k\geq0}\) is P-recursive, and its ordinary generating
function
\[
Q(z)=\sum_{k\geq0}Q_kz^k
\]
is D-finite.
\end{corollary}

\begin{proof}
After replacing the differences of binomial coefficients in
\eqref{eq:Q-multisum} by their rational-times-binomial forms, every summand is
a proper hypergeometric term in \(k,b,c,r\).  The summation ranges are finite
and bounded by affine functions of \(k\).  The closure theorem for finite
proper hypergeometric multisums \cite{PetkovsekWilfZeilberger}, or, equivalently, creative telescoping, therefore
implies that the triple sum is P-recursive in \(k\).  The Catalan boundary term is P-recursive as
well.  The equivalence between P-recursive sequences and D-finite ordinary
generating functions gives the second assertion.
\end{proof}

Define the descent refinement
\begin{equation}\label{eq:Q-refinement}
\Qsf_k(x)=\sum_{w\in\Lat_3^k(321)}x^{\des(w)}.
\end{equation}
The first values are
\begin{align*}
\Qsf_1(x)&=1,\\
\Qsf_2(x)&=1+3x+x^2,\\
\Qsf_3(x)&=1+10x+16x^2+6x^3,\\
\Qsf_4(x)&=1+22x+84x^2+96x^3+30x^4+x^5.
\end{align*}

\begin{remark}
\label{obs:Q-nonreal}
The polynomial $\Qsf_k(x)$ is real-rooted for $1\leq k\leq10$, but not for
$k=11$.  At $k=11$ it has one nonreal conjugate pair
\[
-7.09084909284\pm1.18209112637\,i.
\]
Its coefficient sequence is log-concave and unimodal for every $k\leq25$
in the computations.
\end{remark}

\section{Three symmetric avoidance families}\label{sec:symmetric-families}

The results in this section resolve the three conjectures
in Laudone's final section.  By Theorem~\ref{thm:symmetric-reduction}, it is
enough to describe the corresponding lattice words.

\subsection{Avoidance of \texorpdfstring{$\Orb(1213)$}{Orb(1213)}}
\begin{theorem}\label{thm:1213-terminal}
For $n\geq2$ and $k\geq1$,
\begin{equation}\label{eq:1213-structure}
\Lat_n^k(\Orb(1213))
 =\left\{
 1^k2^k\cdots(n-2)^k v:
 v\in\Lat_2^k\text{ on }\{n-1,n\}
 \right\}.
\end{equation}
Consequently,
\[
L_{n,k}(\Orb(1213))=C_k
\qquad\text{and}\qquad
c_n^k(\Orb(1213))=n!C_k.
\]
\end{theorem}

\begin{proof}
For $n=2$ the assertion is immediate: no pattern in $\Orb(1213)$ can occur,
and the entire word is the terminal binary block in
\eqref{eq:1213-structure}.  Assume $n\geq3$.  Let
$w\in\Lat_n^k$ avoid $\Orb(1213)$.  Suppose that a letter $b>1$ occurs
before the final copy of $1$.  By Lemma~\ref{lem:ordered-occurrences}, the first
copy of $1$ occurs before the selected $b$.  Choose the final $1$ after that
occurrence, and choose $c\notin\{1,b\}$.  By
Lemma~\ref{lem:ordered-occurrences}, the final copy of $c$ occurs after the
final copy of $1$.  Hence $w$ contains a subsequence
\[
1\,b\,1\,c
\]
whose equality pattern is $1213$, a contradiction.  Thus every copy of $1$
precedes every other letter, so $w=1^ku$.

Deleting the initial block and subtracting $1$ from the remaining letters
produces a member of $\Lat_{n-1}^k$ that still avoids $\Orb(1213)$.  Iterating
the argument forces the blocks $1^k,2^k,\ldots,(n-2)^k$, leaving an arbitrary
binary lattice word on $\{n-1,n\}$.

Conversely, consider a word of the form in \eqref{eq:1213-structure}.  If the
repeated letter in a $abac$ occurrence is at most $n-2$, then all
its copies are consecutive, so no different letter can lie between two of
them.  If the repeated letter is $n-1$ or $n$, then the terminal block contains
only two distinct letters and cannot supply the third letter required by
$1213$. Thus the word avoids the orbit.  The binary terminal block is counted
by $C_k$, and the rest follows from
Theorem~\ref{thm:symmetric-reduction}.
\end{proof}

\begin{example}\label{ex:1213-terminal}
For $n=5$ and $k=3$, the word
\[
111\,222\,333\,445455
\]
avoids $\Orb(1213)$. The terminal binary word
$445455$ here corresponds to the Dyck path $UUDUDD$.
\end{example}

\begin{figure}[H]
\centering
\begin{tikzpicture}[x=0.47cm,y=0.52cm,
  box/.style={draw,minimum width=0.42cm,minimum height=0.42cm,inner sep=0pt,font=\small}]
\foreach \x/\a in {0/1,1/1,2/1,3/2,4/2,5/2,6/3,7/3,8/3,9/4,10/4,11/5,12/4,13/5,14/5}
  \node[box] at (\x,4.0) {\a};
\draw[rounded corners] (-0.45,3.55) rectangle (8.45,4.45);
\draw[rounded corners] (8.55,3.55) rectangle (14.45,4.45);
\node at (4,3.05) {forced blocks};
\node at (11.5,3.05) {terminal Dyck block};
\draw[thick,-{Stealth[length=2mm]}] (11.5,2.65)--(11.5,1.8);
\draw[thick] (9,0)--(9.8,0.8)--(10.6,1.6)--(11.4,0.8)--(12.2,1.6)--(13.0,0.8)--(13.8,0);
\draw[dashed] (8.9,0)--(13.9,0);
\node[below] at (11.4,-0.15) {$UUDUDD$};
\end{tikzpicture}
\caption{The terminal-block bijection in Theorem~\ref{thm:1213-terminal}.}
\label{fig:1213-terminal}
\end{figure}
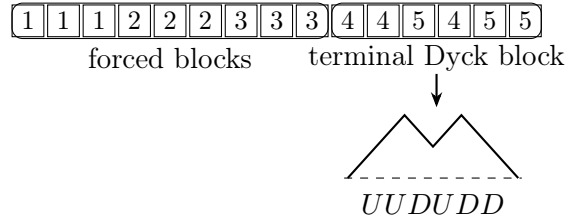

\subsection{Avoidance of \texorpdfstring{$\Orb(12112)$}{Orb(12112)}}
\begin{lemma}\label{lem:binary-12112}
Let $k\geq2$, let $a<b$, and let $u$ be a binary lattice word with $k$ copies of each of
$a,b$.  Then $u$ avoids $\Orb(12112)=\{12112,21221\}$ if and only if
\begin{equation}\label{eq:binary-12112-form}
u=a^{k-1}b^r a b^{k-r}
\qquad\text{for a unique }0\leq r\leq k-1.
\end{equation}
\end{lemma}

\begin{proof}
Suppose that a copy of $b$ occurs before the $(k-1)$st copy of $a$.  A copy of
$a$ occurs before that $b$ because $u$ is ballot, two copies of $a$ occur after
it, and the final $b$ occurs after the final $a$ by
Lemma~\ref{lem:ordered-occurrences}.  These five letters form
$a\,b\,a\,a\,b$, an occurrence of $12112$. Therefore the first $k-1$
letters are $a$'s.  Only one $a$ remains, and the ballot condition allows
$r$ copies of $b$ before it, where $0\leq r\leq k-1$, giving
\eqref{eq:binary-12112-form}.

Conversely, after the first $b$ in a word of the displayed form there is at
most one remaining $a$. Hence neither $a\,b\,a\,a\,b$ nor
$b\,a\,b\,b\,a$ can occur. Uniqueness of $r$ is immediate.
\end{proof}

For $w\in\Lat_n^k$, define $\skel(w)$ by deleting the final occurrence of
each letter. 
\begin{theorem}\label{thm:12112-skeleton}
For $n\geq1$ and $k\geq2$,
\begin{equation}\label{eq:12112-skeleton}
w\in\Lat_n^k(\Orb(12112))
\quad\Longleftrightarrow\quad
\skel(w)=1^{k-1}2^{k-1}\cdots n^{k-1}.
\end{equation}
Equivalently,
\[
p_{i,k-1}(w)<p_{j,1}(w)
\qquad(1\leq i<j\leq n).
\]
\end{theorem}

\begin{proof}
Restrict $w$ to any two letters $i<j$. By
Lemma~\ref{lem:ordered-occurrences}, the restriction is a binary lattice word.
If $w$ avoids $\Orb(12112)$, then so does every restriction, and
Lemma~\ref{lem:binary-12112} gives
$p_{i,k-1}<p_{j,1}$.  In particular, all non-final copies of $1$ precede every
copy of $2$, all non-final copies of $2$ precede every copy of $3$, and so on.
After the final copy of each letter is deleted, the remaining word is therefore
$1^{k-1}2^{k-1}\cdots n^{k-1}$.

Conversely, suppose the displayed skeleton condition holds.  For $i<j$, the
restriction to $\{i,j\}$ begins with $i^{k-1}$.  The final $i$ occurs before
the final $j$ by Lemma~\ref{lem:ordered-occurrences}, so the restriction has the
form $i^{k-1}j^r i j^{k-r}$ with $0\leq r\leq k-1$ and it
avoids both relabelings of $12112$.  Since every occurrence of
$\Orb(12112)$ uses exactly two distinct letters, $w$ avoids the full orbit.
\end{proof}
\begin{theorem}\label{thm:12112-fuss}
For $n\geq1$ and $k\geq2$, there is a bijection
\[
\Lat_n^k(\Orb(12112))
\longleftrightarrow
\FussDyck_{k,n}
\longleftrightarrow
\Tree_n^{(k)},
\]
where $\Tree_n^{(k)}$ denotes full ordered $k$-ary trees with $n$ internal
vertices.  Consequently,
\[
L_{n,k}(\Orb(12112))=\FC_{k,n}
\qquad\text{and}\qquad
c_n^k(\Orb(12112))=n!\FC_{k,n}.
\]
\end{theorem}

\begin{proof}
Given $w$ satisfying Theorem~\ref{thm:12112-skeleton}, read it from left to
right.  Replace every non-final occurrence of a letter by
$U=(1,1)$ and every final occurrence by $D=(1,-(k-1))$.  There are
$(k-1)n$ up-steps and $n$ down-steps. After $d$ final occurrences have
appeared, these are necessarily the final copies of $1,2,\ldots,d$, because
final occurrences are ordered by Lemma~\ref{lem:ordered-occurrences}.  Each of
those letters has already contributed its $k-1$ non-final occurrences.  Thus at
least $(k-1)d$ up-steps have appeared, so the height never becomes negative.
The image is a proper Fuss--Dyck path.

For the converse, take $P\in\FussDyck_{k,n}$.  Label its up-steps, from left to right, by
\[
1^{k-1}2^{k-1}\cdots n^{k-1},
\]
and label its down-steps, from left to right, by $1,2,\ldots,n$.  When the
$i$th down-step occurs, non-negativity after that step implies that at least
$i(k-1)$ up-steps have already occurred. Hence all $k-1$ non-final copies of
$i$ precede its final copy.  The non-final copies are ordered by their labels
and the final copies are likewise ordered, so
Lemma~\ref{lem:ordered-occurrences} shows that the resulting word is a lattice
word.  Its non-final skeleton is the required sorted word and it
avoids $\Orb(12112)$.  The two constructions are inverses of each other.

The standard first-return decomposition of a nonempty Fuss--Dyck path is
\[
U P_1\,U P_2\cdots U P_{k-1}\,D P_k,
\]
where each $P_i$ is again a Fuss--Dyck path.  If
$F(z)=\sum_{n\geq0}\card{\FussDyck_{k,n}}z^n$, then
$F(z)=1+zF(z)^k$. Lagrange inversion gives
\[
[z^n]F(z)=\frac1n[u^{n-1}](1+u)^{kn}
 =\frac1n\binom{kn}{n-1}=\FC_{k,n}.
\]
The bijection to a full ordered
$k$-ary tree is built based on the same first-return decomposition. Finally, apply Theorem~\ref{thm:symmetric-reduction}.
\end{proof}

\begin{example}\label{ex:12112-fuss}
For $n=3$ and $k=3$, the word
\[
w=111223233
\]
has non-final skeleton $112233$.  Marking the final $1$, final $2$, and final
$3$ by down-steps gives
\[
UUDUUUDUD,
\]
whose heights are $1,2,0,1,2,3,1,2,0$.
\end{example}

\begin{figure}[H]
\centering
\begin{tikzpicture}[x=0.58cm,y=0.55cm,
  box/.style={draw,minimum width=0.5cm,minimum height=0.45cm,inner sep=0pt,font=\small},
  final/.style={box,fill=gray!28}]
\node at (4,4.5){shaded boxes mark final occurrences};
\node[box] at (0,3.6){1};
\node[box] at (1,3.6){1};
\node[final] at (2,3.6){1};
\node[box] at (3,3.6){2};
\node[box] at (4,3.6){2};
\node[box] at (5,3.6){3};
\node[final] at (6,3.6){2};
\node[box] at (7,3.6){3};
\node[final] at (8,3.6){3};
\draw[thick] (0,0)--(1,1)--(2,2)--(3,0)--(4,1)--(5,2)--(6,3)--(7,1)--(8,2)--(9,0);
\draw[dashed] (-0.1,0)--(9.1,0);
\node[below] at (4.5,-0.2){$UUDUUUDUD$};
\end{tikzpicture}
\caption{The last-occurrence map of Theorem~\ref{thm:12112-fuss} for
$w=111223233$.  Each down-step has vertical displacement $-2$.}
\label{fig:12112-fuss}
\end{figure}
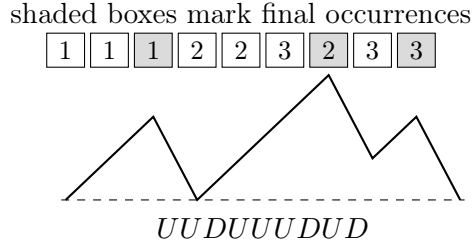

\subsection{Avoidance of \texorpdfstring{$\Orb(1231)$}{Orb(1231)}: A Dyck-decorated matching}

For $w\in\Lat_n^k$, define the span interval
\[
I_i(w)=[p_{i,1}(w),p_{i,k}(w)]
\]
and the \emph{overlap graph} $G(w)$ on vertex set $[n]$ by joining $i<j$ when
\begin{equation}\label{eq:overlap-edge}
p_{j,1}(w)<p_{i,k}(w).
\end{equation}
Since both the first and final occurrence positions increase with the label,
condition \eqref{eq:overlap-edge} is the condition that the two span
intervals overlap.

\begin{lemma}\label{lem:1231-overlap}
A lattice word $w\in\Lat_n^k$ avoids $\Orb(1231)$ if and only if every vertex
of $G(w)$ has degree at most one.  In that case every edge of $G(w)$ joins two
consecutive labels, and therefore $G(w)$ is a matching in the path
$P_n=1-2-\cdots-n$.
\end{lemma}

\begin{proof}
An occurrence $a\,b\,c\,a$ with $a,b,c$ pairwise distinct exists
when two distinct labels occur strictly between the first and final copies of
$a$.  A label $j>a$ occurs in that interval if and only if its first occurrence does, which is equivalent to $p_{j,1}<p_{a,k}$.  A label $j<a$ occurs there if
and only if its final occurrence does, which is equivalent to
$p_{a,1}<p_{j,k}$.  These are exactly the neighbors of $a$ in $G(w)$. Hence
avoidance is equivalent to maximum degree at most one.

Suppose $i<j$ with $j\geq i+2$ and $ij$ is an edge.  Then
\[
p_{i+1,1}<p_{j,1}<p_{i,k},
\]
so $i(i+1)$ is also an edge.  The vertex $i$ would have degree at least two,
contrary to avoidance. Thus every edge joins consecutive labels, and a graph
of maximum degree one with this property is a matching in $P_n$.
\end{proof}

Let
\[
\Dyck_k^\star=\Dyck_k\setminus\{U^kD^k\}
\]
be the non-mountain Dyck paths, equivalently the binary lattice words in which
the two span intervals overlap.

\begin{theorem}\label{thm:1231-matching}
For $n,k\geq1$, there is a canonical bijection
\begin{equation}\label{eq:1231-bijection}
\Lat_n^k(\Orb(1231))
\longleftrightarrow
\bigsqcup_{M\in\Match(P_n)}
\prod_{\{i,i+1\}\in M}\Dyck_k^\star.
\end{equation}
An unmatched vertex $i$ contributes the singleton block $i^k$, while a matched
edge $\{i,i+1\}$ contributes a nontrivial binary lattice word on
$\{i,i+1\}$.  These blocks occur from left to right in increasing label order,
and the decomposition is unique.
\end{theorem}

\begin{proof}
Let $w$ avoid $\Orb(1231)$.  By Lemma~\ref{lem:1231-overlap}, $G(w)$ is amatching in $P_n$.  If $i$ is unmatched, then no other span intersects $I_i$ and no other letter occurs between the first and final copies of $i$,
so these copies form the contiguous block $i^k$.

If $\{i,i+1\}$ is a matched edge, then the union of $I_i$ and $I_{i+1}$
contains no third span: a third intersecting span would give degree at least
two to one of the matched vertices. Therefore the segment from the first
$i$ to the final $i+1$ is a contiguous binary lattice word on
$\{i,i+1\}$.  The intervals overlap, so this binary word is not
$i^k(i+1)^k$ and corresponds to an element of $\Dyck_k^\star$.  Since first
and final occurrence positions are ordered by the labels, these singleton and
paired blocks appear in increasing order.  This constructs the right-hand
side of \eqref{eq:1231-bijection} uniquely.

For the other direction, choose a matching $M$ and concatenate the indicated blocks in
increasing label order.  Within a paired block the binary ballot condition
gives the required adjacent prefix inequality; all earlier blocks are already
complete and all later blocks have not begun. Thus the concatenation lies in
$\Lat_n^k$.  Each letter has at most its matched partner inside its span, so
Lemma~\ref{lem:1231-overlap} implies avoidance.  The constructions are inverses of each other.
\end{proof}

\begin{corollary}\label{cor:1231-recurrence}
Let
\[
a_{n,k}=L_{n,k}(\Orb(1231)),
\qquad a_{0,k}=a_{1,k}=1.
\]
Then
\begin{equation}\label{eq:1231-recurrence}
a_{n,k}=a_{n-1,k}+(C_k-1)a_{n-2,k}
\qquad(n\geq2),
\end{equation}
\begin{equation}\label{eq:1231-sum}
a_{n,k}=
\sum_{j=0}^{\lfloor n/2\rfloor}
\binom{n-j}{j}(C_k-1)^j,
\end{equation}
and
\begin{equation}\label{eq:1231-ogf}
\sum_{n\geq0}a_{n,k}z^n
 =\frac{1}{1-z-(C_k-1)z^2}.
\end{equation}
Therefore
\[
c_n^k(\Orb(1231))=n!a_{n,k}.
\]
\end{corollary}

\begin{proof}
Partition the decorated matchings according to whether the final vertex $n$ is
unmatched or is matched to $n-1$.  The first case leaves a decorated
matching of $P_{n-1}$; the second leaves one of $P_{n-2}$ and has
$\card{\Dyck_k^\star}=C_k-1$ choices for the final edge.  This gives
\eqref{eq:1231-recurrence}.  A matching of size $j$ in $P_n$ can be chosen in
$\binom{n-j}{j}$ ways, which gives \eqref{eq:1231-sum}.  The recurrence gives
\eqref{eq:1231-ogf}, and the canon count follows from symmetric reduction lemma.
\end{proof}

\begin{example}\label{ex:1231-matching}
For $n=5$ and $k=3$, take the matching
$M=\{\{1,2\},\{4,5\}\}$ and decorate both edges by the Dyck word $UUDUDD$.
The corresponding lattice word is
\[
112122\mid333\mid445455.
\]
The middle vertex is a monomer, and the two matched edges are nontrivial
Dyck-decorated dimers.
\end{example}

\begin{figure}[H]
\centering
\begin{tikzpicture}[x=1.25cm,y=0.9cm,
 v/.style={circle,draw,minimum size=6mm,inner sep=0pt}]
\foreach \i in {1,...,5}{\node[v] (v\i) at (\i,1.8){\i};}
\foreach \i/\j in {1/2,2/3,3/4,4/5}{\draw[gray!60] (v\i)--(v\j);}
\draw[very thick,blue!65!black,bend left=35] (v1) to (v2);
\draw[very thick,blue!65!black,bend left=35] (v4) to (v5);
\node[draw,rounded corners,fit=(v1)(v2),inner sep=5pt,label=below:{\small $112122$}] {};
\node[draw,rounded corners,fit=(v3),inner sep=5pt,label=below:{\small $333$}] {};
\node[draw,rounded corners,fit=(v4)(v5),inner sep=5pt,label=below:{\small $445455$}] {};
\end{tikzpicture}
\caption{A Dyck-decorated matching for the word in
Example~\ref{ex:1231-matching}.}
\label{fig:1231-matching}
\end{figure}
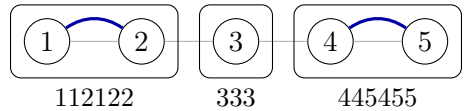

\section{Conjectures and questions}
\label{sec:questions}

All conjectures and questions in this section are open.

\subsection[The sequence Bk]{The sequence \(B_k\)}

Computational evidence suggests the following
real-rootedness conjecture.

\begin{conjecture}
\label{conj:B-real-rooted}
For every \(k\geq1\), the polynomial \(\Bsf_k(x)\) is real-rooted.  
\end{conjecture}

\subsection[The sequence Qk]{The sequence \(Q_k\)}

\begin{conjecture}
\label{conj:Q-nonalgebraic}
The D-finite series \(Q(z)\) is transcendental over \(\Q(z)\).
\end{conjecture}

\begin{conjecture}
\label{conj:Q-logconcave}
For every \(k\geq1\), the coefficients of \(\Qsf_k(x)\) are log-concave and
therefore unimodal.
\end{conjecture}

\begin{question}
\label{ques:Q-structural-model}
Is there a finite-class generating tree, or a model by a fixed number of
classical path or tree objects with local compatibility conditions, whose
natural parameters recover $(b,c,r)$ in Theorem~\ref{thm:Q-multisum}?
\end{question}

\subsection{Continued fractions and moment sequences}

The Narayana component supplies the benchmark
\begin{equation}\label{eq:narayana-sfraction}
\sum_{k\geq0}N_k(x)z^k
 =\cfrac{1}{1-\cfrac{z}{1-\cfrac{xz}{1-\cfrac{z}{1-\cfrac{xz}{\ddots}}}}},
\end{equation}
the classical alternating Stieltjes fraction \cite{Flajolet}.
\begin{remark}[Hankel minors]
For each of the sequences $(B_k)$ and $(Q_k)$, all $3431$ minors of the leading
$7\times7$ Hankel matrix $(a_{i+j})_{0\leq i,j\leq6}$ are positive in exact
arithmetic.
\end{remark}

\begin{conjecture}
\label{conj:moment-property}
Both $(B_k)_{k\geq0}$ and $(Q_k)_{k\geq0}$ are Stieltjes moment sequences;
equivalently, their infinite Hankel matrices are totally nonnegative.
\end{conjecture}

\subsection{Beyond three distinct letters}

\begin{question}
\label{ques:beyond-three}
Which parts of the relabeling and class-reduction mechanism survive for fixed
alphabet size \(n=4\)?  A first target might be to classify those forbidden sets
whose lattice-word components reduce to forced blocks and binary or ternary
Catalan objects, before attempting the complete classification of the \(2^{24}\) forbidden subsets of $\Sym_4$.
\end{question}

\section*{Acknowledgments}
\paragraph{A.I. usage declaration}
Much of the Python and SageMath code that was used in this project was generated by OpenAI's ChatGPT and Google's Gemini. We have also used these tools to improve the writing in many places, and we have also followed the manuscript organisation suggested by Gemini. ChatGPT suggested adding additional mathematical details in some places which we have done. ChatGPT was also used to generate the figures in the article. The mathematical ideas in the paper are the author's own.

\end{document}